\documentclass{amsart}
\usepackage{array}
\usepackage{amssymb, amsthm}
\usepackage{ulem}
\usepackage{amsmath}
\usepackage{url}
\usepackage{hyperref}
\usepackage{enumitem}

\newtheorem{keythm}{Theorem}

\newtheorem{keyprop}[keythm]{Proposition}
{\theoremstyle{definition} }

\newtheorem{thm}{Theorem}[section]
\newtheorem{cor}[thm]{Corollary}
\newtheorem{lem}[thm]{Lemma}
\newtheorem{prop}[thm]{Proposition}
{\theoremstyle{definition} \newtheorem{defn}[thm]{Definition} }
{\theoremstyle{remark} }
{\theoremstyle{remark} }

\newtheorem{bigclm}[thm]{Claim}
{\theoremstyle{remark} }

 \newcommand{\RR}{\mathbb{R}}
\newcommand{\CC}{\mathbb{C}} 
 
\newcommand{\NN}{\mathbb{N}} 
 
\def\Kopmore{\mathbf{K}_{t, a}}

\newcommand{\disp}[0]{\displaystyle}
\newcommand{\bracepair}[1]{\left\lbrace #1 \right\rbrace}
\newcommand{\anglepair}[1]{\left\langle #1 \right\rangle}
\newcommand{\parenpair}[1]{\left( #1 \right)}
\newcommand{\vertpair}[1]{\left\vert #1 \right\vert}

\renewcommand{\Re}{\operatorname{Re}}

\newcommand{\thisp}{p}
\newcommand{\thisq}{q}

\newcommand{\oCd}{m}

\newcommand{\ourvarT}{t}
\newcommand{\ourintT}{T}
\newcommand{\ourfracT}{\tau}
\newcommand{\ourvarR}{s}
\newcommand{\ourintR}{S}
\newcommand{\ourfracR}{\sigma}
\newcommand{\genercpx}{\zeta}

\newcommand{\mrep}{\mu}
\newcommand{\setn}{\overline{n}}
\newcommand{\risingfac}[2]{(#1)_{#2}}
\newcommand{\xvec}{\mathbf{x}}
\newcommand{\yvec}{\mathbf{y}}
\newcommand{\cvec}{\mathbf{c}}

\begin{document}
\title{On the Positive-Definiteness of an Anisotropic Operator}
\author{Charles E. Baker}
\address{%
Department of Mathematics\\
The Ohio State University\\
Columbus, OH 43210}
\email{\url{baker.1656@osu.edu}}



\begin{abstract}
We study the positive-definiteness of a family of $L^2(\RR)$ integral operators with kernel $K_{t, a}(x, y) = ( 1 + (x-y)^2 + a(x^2 + y^2)^t)^{-1}$, for $t > 0$ and $a > 0$.  For $0 < t \leq 1$ and $a > 0$, the known theory of positive-definite kernels and conditionally negative-definite kernels confirms positive-definiteness.  For $t > 1$ and $a$ sufficiently large, the integral operator is \textit{not} positive-definite.  For $t$ not an integer, but with integer odd part, the integral operator is \textit{not} positive-definite.  
\end{abstract}
\maketitle
\normalem

\section{Introduction}

In this paper, we study the integral operators $\Kopmore$ defined by

\begin{equation}
\Kopmore [c](x) = \int_{\RR} K_{\ourvarT, a} (x, y) c(y) \, dy , \quad c \in L^2(\RR),
\label{eq:intta}
\end{equation}
where 
$$ K_{\ourvarT, a} (x, y) = \frac{1}{\pi} \cdot \frac{1}{1 + (x - y)^2 + a(x^2 + y^2)^{\ourvarT}} , \, t > 0,\, a > 0.$$
$\Kopmore$ is a bounded, compact operator from $L^2 (\RR)$ to itself for $\ourvarT > 0$ and $a > 0$.  This is a straightforward verification using the basic facts of compact operators and the Schur Test.

This operator (with $\ourvarT = 2$) was considered by P. Krotkov and A.\\
Chubukov in the papers \cite{KrotChubOne}, \cite{KrotChubTwo} as a part of their simplified model for high-tempera-ture superconductivity.  The asymptotics of the largest eigenvalue of this operator, around $a = 0$, were studied in the papers \cite{MitAn}, \cite{MitSob}, and \cite{Adduci}.  Examining an open question stated in Section 6.2 of \cite{MitSob}, we wish to determine for which $\ourvarT > 0$ and $a > 0$ the operator $\Kopmore$ is \textit{positive-definite}; that is, we wish to determine for which $\ourvarT > 0$ and $a > 0$ the inequality
\[
\anglepair{\Kopmore[c], c} = \iint\limits_{\RR \times \RR} K_{\ourvarT, a}(x, y) c(x) \overline{c(y)} \, dy \, dx \geq 0
\]
holds for all $c \in L^2(\RR)$.

By the continuity of the kernel $K_{\ourvarT, a}$ and the boundedness of the operator $\Kopmore$, positive-definiteness of $\Kopmore$ is equivalent to the statement that
\begin{equation}
\sum_{j = 1}^n \sum_{k = 1}^n c_j \bar{c}_k K_{\ourvarT, a}(x_j, x_k) \geq 0 \label{eq:PDKFta}
\end{equation}
for all $n \in \NN$, $\xvec = (x_1, \dotsc, x_n) \in \RR^n$, and $\cvec  = (c_1, \dotsc, c_n) \in \CC^n$.  

More generally, for any nonempty set $X$, a function $K: X \times X \to \CC$ is called a \textit{positive-definite kernel} if 
\begin{equation}
\sum_{j = 1}^n \sum_{k = 1}^n c_j \bar{c}_k K(x_j, x_k) \geq 0 \label{eq:PDKF}
\end{equation}
holds for all $n \in \NN$, $\xvec = (x_1, \dotsc, x_n) \in X^n$, and $\cvec = (c_1, \dotsc, c_n) \in \CC^n$.  

In this paper, we prove the following.
\begin{keyprop}
For $0 < \ourvarT \leq 1$ and $a > 0$, $K_{\ourvarT, a}$ is a positive-definite kernel; hence, $\Kopmore$ is a positive-definite operator. \label{prop:smallt} 
\end{keyprop}

\begin{keyprop}
If $\ourvarT > 1$, and $\disp a >  a_0(\ourvarT) := \frac{2^{\ourvarT^2 - 1} + 2^{\ourvarT^2 - t}}{\left( 2^{\ourvarT-1}-1 \right)^{(2\ourvarT - 1)}} $,\vspace{6 pt}  then $K_{\ourvarT, a}$ is \uline{not} a positive-definite kernel; hence, $\Kopmore$ is \uline{not} a positive-definite operator.  \label{prop:biga}
\end{keyprop} 

\begin{keythm}
If $\ourvarT \in \bigcup_{k \in \NN} (2k - 1, 2k)$, and if $a > 0$, then $K_{\ourvarT, a}$ is \uline{not} a positive-definite kernel; hence, $\Kopmore$ is \uline{not} a positive-definite operator. \label{thm:oddt}
\end{keythm}

We prove these statements in Sections~\ref{sec:CNDK}, \ref{sec:Necessary}, and \ref{sec:mainsec}, respectively.  

\section{Conditionally Negative-Definite Kernels: the case \texorpdfstring{$0 < t \leq 1$}{0 < t <= 1}} \label{sec:CNDK}

\begin{defn}  Let $X$ be a nonempty set. A function $K: X \times X \to \CC$ is a \textit{conditionally negative-definite kernel} if it is Hermitian (that is, for all $x$ and $y$ in $X$, $K(y, x) = \overline{K(x, y)} $) and satisfies
\begin{equation}
\sum_{j = 1}^n \sum_{k = 1}^n c_j \bar{c}_k K(x_j, x_k) \leq 0
\label{eq:CND}
\end{equation}
for all $n \in \NN$, $n \geq 2$,  $\xvec = (x_1, \dotsc, x_n) \in X^n$, and $\cvec = (c_1, \dotsc, c_n) \in \CC^n$ with $\disp \sum_{j = 1}^n c_j = 0$.  
\end{defn}

Our interest in conditionally negative-definite kernels stems from the following connection to positive-definite kernels.  Let $\CC_+ := \bracepair{\zeta \in \CC : \Re \zeta  \geq 0}$.
\begin{prop}[\cite{BCR}, p. 75]
$K: X \times X \to \CC_+$ is a conditionally negative-definite kernel if and only if, for all $r > 0$, 
\[ \disp
\frac{1}{r + K(x, y)}
\] is a positive-definite kernel. 
\label{prop:inverse}
\end{prop}

Therefore, if $(x-y)^2 + a(x^2 + y^2)^{\ourvarT}$ is a conditionally negative-definite kernel for some $\ourvarT > 0$ and $a > 0$, then by the $r = 1$ case of Proposition~\ref{prop:inverse},\vspace{6pt} \\ $\disp \frac{1}{1 + (x-y)^2 + a (x^2 + y^2)^{\ourvarT}}$ \vspace{6pt} is a positive-definite kernel, and hence $ K_{\ourvarT, a}$  is a positive-definite kernel.  To prove Proposition~\ref{prop:smallt}, it therefore suffices to demonstrate the following.
\begin{bigclm}
If $0 < \ourvarT \leq 1$ and $a > 0$, then $(x-y)^2 + a(x^2 + y^2)^{\ourvarT}$ is a conditionally negative-definite kernel.
\label{claim:CNDcases}
\end{bigclm}

To prove this, we use the following fact.
\begin{prop}[\cite{BCR}, Corollary 3.2.10]
If $K: X \times X \to \CC$ is a conditionally negative-definite kernel and satisfies $K(x, x) \geq 0$ for all $x \in X$, then $K^{\ourvarT}$ is a conditionally negative-definite kernel for all $\ourvarT$ such that $0 < \ourvarT \leq 1$.  
\label{prop:CNDReduce}
\end{prop}

\begin{proof}[Proof of Claim~\ref{claim:CNDcases}]
First, $(x - y)^2$ is a conditionally negative-definite kernel on $\RR \times \RR$: see \cite{BCR}, Section 3.1.10.  Therefore, fixing $\ourvarT$, $0 < \ourvarT \leq 1$, and $a > 0$, it suffices to show that $(x^2 + y^2)^{\ourvarT}$ is a conditionally negative-definite kernel on $\RR \times \RR$, since the class of conditionally negative-definite kernels is closed under addition and positive scalar multiplication (\cite{BCR}, 3.1.11).

Moreover, for any function $f:X \to \CC$, $f(x) + \overline{f(y)}$ is a conditionally negative-definite kernel; see \cite{BCR}, Section 3.1.9.  In particular, $x^2 + \overline{y^2} = x^2 + y^2$ is a conditionally negative-definite kernel on $\RR \times \RR$. We may invoke Proposition~\ref{prop:CNDReduce} , with $K(x, y) = x^2 + y^2$, because $x^2 + x^2 \geq 0$ whenever $x \in \RR$. Therefore, $(x^2 + y^2)^{\ourvarT}$ is a conditionally negative-definite kernel for all $\ourvarT$, $0 < \ourvarT \leq 1$, as required.   
\end{proof}

\section{A necessary condition: the case \texorpdfstring{$\ourvarT > 1$, $a$}{t > 1, a} large} \label{sec:Necessary}

The $n = 2$ case of \eqref{eq:PDKFta} implies that for all $x, y \in \RR$,

\begin{equation}
\vert K_{\ourvarT, a}(x, y) \vert^2 \leq K_{\ourvarT, a}(x, x) K_{\ourvarT, a}(y, y) \label{eq:Schwarz}
\end{equation}

(see \cite{BCR}, Section 3.1.8), or

\begin{align*}
&\left(\frac{1}{\pi} \cdot \frac{1}{1 + (x-y)^2 + a(x^2 + y^2)^{\ourvarT}}\right)^2\\
& \leq \parenpair{\frac{1}{\pi} \cdot \frac{1}{1 + (x - x)^2 + a(x^2 + x^2)^{\ourvarT}} } \cdot \parenpair{\frac{1}{\pi} \cdot \frac{1}{1 + (y - y)^2 + a(y^2 + y^2)^{\ourvarT}}} .
\end{align*}

We may rewrite this inequality as

\begin{align} \begin{aligned}
0 & \leq & & (1 + (x-y)^2)^2 - 1 \\
&  & &+ 2a \big( (1 + (x-y)^2)(x^2 + y^2)^{\ourvarT} - 2^{\ourvarT-1} ((x^2)^{\ourvarT} + (y^2)^{\ourvarT}) \big)\\
& & &+ a^2 \big((x^2 + y^2)^{2\ourvarT} - (4 x^2 y^2)^{\ourvarT} \big). \label{eq:SchSpecific}
\end{aligned}
\end{align}

We now prove Proposition~\ref{prop:biga} by refuting a special case of \eqref{eq:SchSpecific}.

\begin{proof}[Proof of Proposition~\ref{prop:biga}]
Suppose that for some $\ourvarT > 1$ and $a > 0$, $K_{\ourvarT, a}$ is positive-definite. \eqref{eq:SchSpecific}, with $y = 0$, implies that 
\[
0 \leq (1 + x^2)^2 - 1  + 2a(x^2)^{\ourvarT}\big((1 + x^2) - 2^{\ourvarT-1} \big) + a^2 (x^2)^{2\ourvarT},
\]
or, letting $z = x^2 \geq 0$,
\begin{equation}
0 \leq g(z; \ourvarT, a) := (1 + z)^2 - 1 + 2az^{\ourvarT}\big((1 + z) - 2^{\ourvarT-1}\big) + a^2 z^{2\ourvarT}. \label{eq:a0det}
\end{equation}

Note that unless $(1 + z) < 2^{\ourvarT-1}$, $g(z; \ourvarT, a)$ must be nonnegative for nonnegative $z$.  Note also that $g(0; \ourvarT, a) = 0$.  Therefore, we may assume that $0 < z < 2^{\ourvarT - 1} - 1$.

We minimize $g(z; \ourvarT, a)$ in the variable $a$: the minimizing value of $a$ is
\begin{equation}
a =\widetilde{a}_{\ourvarT}(z) : = \frac{\big(2^{\ourvarT-1} -  (1 + z)\big)}{z^{\ourvarT}}, \quad \text{ or } \quad  az^{\ourvarT} = \big(2^{\ourvarT-1} -  (1 + z)\big), \label{eq:earlyspeca}
\end{equation}
and the minimum value of $g(z; \ourvarT, a)$ in $a$ is 
\[
g(z; \ourvarT, \widetilde{a}_{\ourvarT} (z)) = (1 + z)^2 - 1 - \big(2^{\ourvarT-1} - ( 1 + z)\big)^2 = 2^{\ourvarT} z - (2^{\ourvarT-1} - 1)^2.
\]
Therefore, we see that $g(z; \ourvarT, \widetilde{a}_{\ourvarT} (z)) < 0$ if $\disp 0 < z < z_0 := \frac{(2^{\ourvarT-1} - 1)^2}{2^{\ourvarT}}$ \vspace{4 pt}; this contradicts \eqref{eq:a0det} for $\ourvarT > 1$ and $a = \widetilde{a}_{\ourvarT} (z)$.  We now determine the set of $a$ that can be written as $\widetilde{a}_{\ourvarT}(z)$ for $0 < z < z_0$.

\begin{lem}
For $\ourvarT > 1$, the range of $\widetilde{a}_{\ourvarT}$ on $(0, z_0)$ is $(\widetilde{a}_{\ourvarT}(z_0), \infty)$.
\end{lem}

\begin{proof}
We first show that for $\ourvarT > 1$, $\widetilde{a}_{\ourvarT} (z)$ is continuous and strictly decreasing for $z \in (0, z_0)$.  Indeed, the continuity and differentiability of $\widetilde{a}_{\ourvarT} (z)$, $z > 0$, is evident from \eqref{eq:earlyspeca}, and the derivative in $z$ is
\begin{align*}
\frac{d}{dz} \widetilde{a}_{\ourvarT}(z) &= \frac{-z^{\ourvarT}  - \big(2^{\ourvarT-1} - (1 + z)\big) \ourvarT z^{\ourvarT-1}}{z^{2\ourvarT}},
\end{align*}
which is negative if $ 0 < z < 2^{\ourvarT-1} - 1$.  Since $0 < z < z_0$, and
\[
z_0 = (2^{\ourvarT-1} - 1) \cdot \frac{(2^{\ourvarT-1} - 1)}{2^{\ourvarT}} < \frac{1}{2}(2^{\ourvarT-1} - 1),
\]
$(0, z_0) \subset (0, 2^{\ourvarT-1} - 1)$, so $\widetilde{a}_{\ourvarT}(z)$ is continuous and strictly decreasing on $(0, z_0)$.  

Moreover, for all $\ourvarT > 1$, $\disp \lim_{z \to 0^+} \widetilde{a}_{\ourvarT} (z) = \infty$.  Therefore, as $z$ decreases from $z_0$ to $0$, $\widetilde{a}_{\ourvarT} (z)$ increases continuously from $\widetilde{a}_{\ourvarT} (z_0)$ to $\infty$.  
\end{proof}

Calculating the bounding value of $a$, we determine that
\[ a_0(\ourvarT) := \widetilde{a}_{\ourvarT}(z_0) = \frac{2^{\ourvarT^2 - 1} + 2^{\ourvarT^2 - t}}{(2^{\ourvarT-1} - 1)^{(2\ourvarT - 1)}}, \]
so for all $\ourvarT > 1$ and $a > a_0(\ourvarT)$, $K_{\ourvarT, a}$ is not positive-definite.
\end{proof}

We note that the above argument only studies a very special case of the $n = 2$ condition for positive-definiteness.  For $a \leq a_0(t)$, the positive-definiteness of $a$ is in general undetermined.  We now proceed to rule out positive-definiteness for more $(t, a)$ pairs.

\section{An asymptotics argument: the case \texorpdfstring{$\ourvarT \not\in \NN$, $\lfloor \ourvarT \rfloor$ odd}{t not in N, integer part odd}} \label{sec:mainsec}
\subsection{Rewriting \texorpdfstring{\eqref{eq:PDKFta}}{(2)} to permit an asymptotics argument} \label{subsect:methodintro}
To describe another obstruction to positive-definiteness, we adjust \eqref{eq:PDKFta}.  First, note that since $K_{t, a}$ is a real-valued and symmetric kernel (i.e., $K_{t, a}(y, x) = K_{t, a}(x, y)$), in \eqref{eq:PDKFta}, we may take the $c_j$ to be real as well; see \cite{BCR}, Section 3.1.6.  

Adding together the fractions $\disp c_j \overline{c_k} K_{t, a}(x_j, x_k) = \frac{c_j c_k}{1 + (x_j - x_k)^2 + a(x_j^2 + x_k^2)^t}$ in \eqref{eq:PDKFta}, we note that the resulting denominator is positive, since each summand's denominator is positive.  We thus see that $K_{t, a}$ is a positive-definite kernel if and only if for all $n \in \NN$, for all $\xvec = (x_1, \dotsc, x_n) \in \RR^n$, and for all $\cvec = (c_1, \dotsc, c_n) \in \RR^n$,
\begin{equation}
\sum_{j = 1}^n \sum_{k = 1}^n c_j {c_k} \left( \prod_{\substack{(p, q) \in \bracepair{1, 2, \dotsc , n}^2 \\ (p, q) \neq (j, k) }} \left( 1 + (x_{p} - x_q)^2 + a(x_p^2 + x_q^2)^{\ourvarT}  \right)  \right) \geq 0.
\label{eq:numer}
\end{equation}
To separate the terms by homogeneity, for any fixed $x > 0$, define $y_j$ by $x_j = y_j x$.  By the squaring, in all terms we achieve $x^2 =: z > 0$, so we see that we wish to study the positivity of
\begin{equation}
\disp \sum_{j = 1}^n \sum_{k = 1}^n c_j {c_k} \left( \prod_{\substack{(p, q) \in \bracepair{1, 2, \dotsc , n}^2 \\ (p, q) \neq (j, k) }} \left(1 + (y_{p} - y_q)^2 z + a(y_p^2 + y_q^2)^{\ourvarT} z^{\ourvarT} \right) \right). \label{eq:main2}
\end{equation}
For fixed $n$ in $\NN$, $\yvec = (y_1, \dotsc, y_n) \in \RR^n$, and $\cvec = (c_1, \dotsc, c_n) \in \RR^n$, we therefore define a function from $ z \in \RR^+$ to $\RR$,
\begin{align}
\begin{aligned}
&f\left(z; n, \yvec, \cvec \right) \\
:= &\sum_{j = 1}^n \sum_{k = 1}^n c_j {c_k} \left( \prod_{\substack{(p, q) \in \bracepair{1, 2, \dotsc , n}^2 \\ (p, q) \neq (j, k) }} \left(1 + (y_{p} - y_q)^2 z + a(y_p^2 + y_q^2)^{\ourvarT} z^{\ourvarT} \right) \right),
\end{aligned} \label{eq:fullz}
\end{align}
and the positive-definiteness of $K_{t, a}$ requires that this function is always nonnegative on $(0, \infty)$, for all choices of the parameters $n$, $\yvec$, and $\cvec$. 

We see that after writing out the products and collecting terms with like powers of $z$, $f$ admits a representation as a finite sum of the form
\begin{equation} f\left(z; n, \yvec, \cvec \right) = \sum_{k = 1}^{\nu(n)} b_k\!\!\left( n, \yvec, \cvec \right) \cdot z^{r_k}, 
\label{eq:fullpowers}
\end{equation}
where $r_0 < r_1 < \dotsb < r_{\nu(n)}$ is an enumeration of the distinct powers of $z$ in $f$.  If we can find parameters $n$, $\yvec$, and $\cvec$ such that the term of smallest degree in $z$ has negative coefficient, then by elementary asymptotics, for $z$ positive and small, $f(z; n, \yvec, \cvec)$ will be negative, and so $K_{t, a}$ will not be positive-definite.  

Therefore, we organize the remainder of our paper as follows.  Fix $\ourvarT \not\in \NN$.  In Section~\ref{subsect:intpows}, we will choose a set of parameters $n$, $\yvec$, and $\cvec$ such that all coefficients of terms of degree less than $\ourvarT$ in $f$ are zero.  In Section~\ref{subsect:tpows}, we will show that for the same parameters, the coefficient of $z^{\ourvarT}$  is nonzero, and if the integer part of $\ourvarT$ is odd, it will be negative.  In Section~\ref{subsect:conclude}, we will complete the proof.

Before continuing, however, we note some conventions.  The decomposition of $t$ into its integer and fractional parts will be denoted $\ourvarT = \ourintT + \ourfracT$, where $\ourintT = \lfloor \ourvarT \rfloor$ is the greatest integer less than or equal to $\ourvarT$, and $\ourfracT \in [0, 1)$ is the fractional part.  The set $\bracepair{1, 2, \dotsc, n}$ is denoted $\setn$.  For any set $Y$, the set of $p$-element subsets of $Y$ is denoted $\disp \binom{Y}{p}$; if $Y \neq \emptyset$, $\disp \binom{Y}{0} = \bracepair{\emptyset}$, the set whose only element is the empty set.  Finally, we denote the $n$th rising factorial of $\alpha \in \RR$ as $
\risingfac{\alpha}{n} = (\alpha)(\alpha + 1) \dotsc (\alpha + n - 1)$, for $ n \in \NN $.

\subsection{Removing integer powers of \texorpdfstring{$z$ in $f$}{z in f}} \label{subsect:intpows}  For $\ourvarT$ not in $\NN$, we now study the coefficients of terms in $f$ of degree less than $\ourvarT$.  Since all terms in $f$ are products of terms of degree $0$, $1$, or $\ourvarT$, any term of degree less than $\ourvarT$ must be the product of degree $0$ and $1$ terms, and hence must be a nonnegative integer.  Therefore, to study coefficients of terms of degree less than $\ourvarT$ in $f$, we study the same powers in the following polynomial in $z$,
\begin{equation}
\sum_{j = 1}^n \sum_{k = 1}^n c_j {c_k} \left( \prod_{\substack{(p, q) \in \setn^2 \\ (p, q) \neq (j, k) }} \left(1 + (y_{p} - {y_{q}})^2 z \right) \right),
\label{eq:pureint}
\end{equation}
rather than the full expression in \eqref{eq:fullz}.  We therefore will prove the following statement about the above polynomial.  

\begin{bigclm}\label{clm:noint}
Fix $n \in \NN$, and fix $\oCd \in \NN \cup \bracepair{0}$, $m \leq n^2 - 1$.  If there exist $\yvec = (y_1, \dotsc, y_n) \in \RR^n$ and $\cvec = (c_1, \dotsc, c_n) \in \RR^n$ such that 
$$\sum_{j = 1}^n c_j y_j^{\ell} = 0\, \text{ for all } \, \ell \in \bracepair{0, 1, \dotsc, \oCd},$$
then the coefficient of $z^{\oCd}$ in \eqref{eq:pureint} is zero.
\end{bigclm} 
The small-$\oCd$ cases of the above yield the following result, which we will use in the sequel.
\begin{cor}
Fix $\ourvarT >0$, $\ourvarT \not\in \NN$, and write $\ourvarT = \ourintT + \ourfracT$ with $\ourintT \in \NN \cup \bracepair{0}$ and $\ourfracT \in (0, 1)$.  Also, fix $a > 0$.  If for some $n \in \NN$, there exist $\yvec = (y_1, \dotsc, y_n) \in \RR^n$ and $\cvec = (c_1, \dotsc, c_n) \in \RR^n$ such that
\begin{equation} \label{eq:pointcond}
\disp \sum_{j = 1}^n c_j y_j^{\ell} = 0\, \text{ for all }\, \ell \in \bracepair{0, 1, \dotsc, \ourintT} ,
\end{equation} then the the coefficients of $z^0, z^1, \dotsc z^{\ourintT}$ in $f$ are all zero, and so the smallest-power term in $f$ is the $z^{\ourvarT}$ term.
\label{cor:point}
\end{cor}

\begin{proof}
Our hypothesis is sufficient to use the cases $\oCd = 0, 1, \dotsc, \ourintT$ of \\
Claim~\ref{clm:noint}, so the $z^0, z^1, \dotsc, z^{\ourintT}$ terms in \eqref{eq:pureint} have zero coefficients.  Since the coefficients of these powers are the same in \eqref{eq:pureint} and \eqref{eq:fullz}, we see that these powers have zero coefficients in $f$.  
\end{proof}

To prove Claim~\ref{clm:noint}, we first describe the $z^{\oCd}$ term in \eqref{eq:pureint}.  To create the $(j, k)$-th term's contribution to the $z^{\oCd}$ coefficient, we collect all products of $\oCd$ distinctly-indexed terms of the form $(y_p - y_q)^2$, $(p, q) \in \setn^2 \setminus \bracepair{(j, k)}$, add them together, and then multiply by $c_j c_k$.  The total $z^{\oCd}$ coefficient is therefore

\begin{equation}
\sum_{j = 1}^n \sum_{k = 1}^n c_j {c_k} \left(\sum_{J \in \binom{\setn^2\setminus \bracepair{(j, k)}}{\oCd}} \prod_{(p, q) \in J} (y_p - {y_q})^2  \right).
\label{eq:zmcontrb}
\end{equation}


We would like to make the summation in $J$ independent of $j$ and $k$, so that we can move terms out of the sum in $j$ and $k$.  We get the following.

\begin{lem}\label{lem:rederiveone}
For $n \in \NN$, for any $\yvec = (y_1, \dotsc, y_n) \in \RR^n$, for any $(j, k) \in \setn^2$, and for any $\oCd \in \NN \cup \bracepair{0}$, $0 \leq \oCd \leq n^2- 1$, 
\begin{align}
\begin{aligned}
&\left( \sum_{J \in \binom{\setn^2\setminus \bracepair{(j, k)}}{\oCd}} \prod_{(p, q) \in J} (y_p - {y_q})^2  \right)\\
& =\sum_{v = 0}^{\oCd} (-1)^{v} (y_j - y_k)^{2v} \left(\sum_{J \in \binom{\setn^2}{\oCd-v}} \prod_{(p, q) \in J} (y_p -{y_q})^2  \right).
\end{aligned} \label{eq:univ}
\end{align}
\end{lem}

\begin{proof}
We prove the lemma by induction on $\oCd$.  For $\oCd = 0$, by the convention that the empty product is $1$, both sides are equal to 
$$\left( \sum_{J \in \bracepair{\emptyset}} \prod_{(p, q) \in \emptyset} (y_p -{y_q})^2 \right) = 1,$$
so the initial case holds.  

For the inductive argument, suppose that the statement is true for $\oCd = \mrep$; we wish to prove it for $\oCd = \mrep + 1$. We start with the left-hand side of \eqref{eq:univ} for $\oCd = \mrep + 1$,
$$\sum_{J \in \binom{\setn^2\setminus \bracepair{(j, k)}}{\mrep + 1}} \prod_{(p, q) \in J} (y_p -{y_q})^2,$$
and add and subtract terms so that, in the primary term, the sum in $J$ is independent of $j$ and $k$. The only missing terms are $(\mrep+1)$-fold products with distinct indices, in which one index is $(j, k)$; the other $m$ terms must therefore be members of $\setn^2 \setminus \bracepair{(j, k)}$.  Thus, we have

\begin{gather*}
\left(\sum_{J \in \binom{\setn^2\setminus \bracepair{(j, k)}}{\mrep + 1}} \prod_{(p, q) \in J} (y_p -{y_q})^2  \right) = \left(\sum_{J \in \binom{\setn^2}{\mrep + 1}} \prod_{(p, q) \in J} (y_p -{y_q})^2  \right) \\
 \quad \quad \qquad- (y_j - y_k)^2  \left(\sum_{J \in \binom{\setn^2\setminus \bracepair{(j, k)}}{\mrep}} \prod_{(p, q) \in J} (y_p -{y_q})^2  \right).
\end{gather*} 

Yet the second term above is $- (y_j - y_k)^2$ times the left-hand side of \eqref{eq:univ} (for $\oCd = \mrep$), so by the inductive hypothesis, we have
\begin{align}
\begin{aligned}
&\left(\sum_{J \in \binom{\setn^2}{\mrep + 1}} \prod_{(p, q) \in J} (y_p -{y_q})^2  \right) \\
&\quad + \sum_{v = 0}^{\mrep} (-1)^{v + 1} (y_j - y_k)^{2 (v + 1)}\left(\sum_{J \in \binom{\setn^2}{\mrep -v}} \prod_{(p, q) \in J} (y_p -{y_q})^{2}  \right).
\end{aligned} \label{eq:interim} 
\end{align}

Note that $\disp  \left(\sum_{J \in \binom{\setn^2}{\mrep + 1}} \prod_{(p, q) \in J} (y_p -{y_q})^2  \right)$ \vspace{3pt} becomes the $v = -1$ term of the series in  \eqref{eq:interim}.  Then shifting the index by $1$, we have finished the inductive step.
\end{proof}

By Lemma~\ref{lem:rederiveone}, the coefficient of $z^m$ in \eqref{eq:pureint}, namely \eqref{eq:zmcontrb}, becomes
\begin{align}
\sum_{j = 1}^n \sum_{k = 1}^n c_j {c_k} \left( \sum_{v = 0}^{\oCd} (-1)^{v} (y_j - y_k)^{2v} \left(\sum_{J \in \binom{\setn^2}{\oCd -v}} \prod_{(p, q) \in J} (y_p -{y_q})^2  \right) \right)\\
= \sum_{v = 0}^{\oCd} \left(\left(\sum_{J \in \binom{\setn^2}{\oCd -v}} \prod_{(p, q) \in J} (y_p -{y_q})^2  \right) \sum_{j = 1}^n \sum_{k = 1}^n (-1)^{v} c_j {c_k} (y_j - y_k)^{2 v} \right). \label{eq:intfinal}
\end{align}

Therefore, we see that the coefficient of $z^m$ in \eqref{eq:pureint} is a linear combination of the terms $\disp \sum_{j = 1}^n \sum_{k = 1}^n (-1)^{v} c_j c_k (y_j - y_k)^{2 v}$.  To control these terms, we use the following lemma.

\begin{lem}
For $v\in \NN \cup \bracepair{0}$ and $n \in \NN$, if $\yvec = (y_1, \dotsc, y_n) \in \RR^n$ and $\cvec = (c_1, \dotsc, c_n) \in \RR^n$ satisfy  
\begin{equation} \label{eq:revisedsize}
\sum_{j = 1}^n c_j y_j^u = 0, \quad \text{ for all } u \in \bracepair{0, 1, \dotsc, v},
\end{equation}
then 
\begin{equation}(-1)^{v} \sum_{j = 1}^n \sum_{k = 1}^n c_j {c_k} (y_j - y_k)^{2 v} = 0.
\label{eq:criteq}
\end{equation}
\label{lem:negsums}
\end{lem}

\begin{proof}
We use the binomial formula to expand the desired expression.
\begin{align*}
& (-1)^v \sum_{j = 1}^n \sum_{k = 1}^n c_j {c_k} (y_j - y_k)^{2 v} \\
 = \,&(-1)^v \sum_{j = 1}^n \sum_{k = 1}^n c_j {c_k} \sum_{w = 0}^{2v} \binom{2v}{w}(y_j)^w ( - y_k)^{2v - w}\\
 = \,& \sum_{w = 0}^{2v} (-1)^{ v + w} \binom{2v}{w} \left( \sum_{j = 1}^n c_j y_j^w \right) \cdot \left( \sum_{k = 1}^n c_k y_k^{2v - w} \right).
\end{align*}

If $w \leq v$, then by \eqref{eq:revisedsize},  $\disp \sum_{j = 0}^n c_j y_j^w = 0$, and the sum in $j$ becomes zero.  Similarly, if $w \geq v$, then $2v - w \leq v$, so
$\disp \sum_{k = 0}^n c_k y_k^{2v - w} = 0$, and the sum in $k$ becomes zero.  Therefore, all terms become zero.
\end{proof}

We now are able to finish the proof of Claim~\ref{clm:noint}.

\begin{proof}[Proof of Claim~\ref{clm:noint}]
Fix $n \in \NN$, and let $\oCd \in \NN \cup \bracepair{0}$, $0 \leq \oCd \leq n^2 - 1$.  Suppose that $\yvec = (y_1, \dotsc, y_n) \in \RR^n$ and $\cvec = (c_1, \dotsc, c_n) \in \RR^n$ satisfy the following statements:
\begin{equation}
\sum_{j = 1}^n c_j y_j^{\ell} = 0 \text{ for all } \ell \in \bracepair{0, 1, \dotsc, \oCd}.
\label{eq:hypoth}
\end{equation}

By \eqref{eq:intfinal}, the coefficient of $z^{\oCd}$ in  \eqref{eq:pureint} is
\[
 \sum_{v = 0}^{\oCd} \left(\left(\sum_{J \in \binom{\setn^2}{\oCd -v}} \prod_{(p, q) \in J} (y_p -{y_q})^2  \right) \sum_{j = 1}^n \sum_{k = 1}^n (-1)^{v} c_j {c_k} (y_j - y_k)^{2 v} \right). \]
Yet by $v \leq \oCd$, the statements in \eqref{eq:hypoth} assure us that $$\sum_{j = 1}^n c_j y_j^{u} = 0$$
 for all $u \in \bracepair{0, 1, \dotsc, v}$.  Invoking Lemma~\ref{lem:negsums}, we see that
$$ (-1)^{v} \sum_{j = 1}^n \sum_{k = 1}^n c_j {c_k} (y_j - y_k)^{2 v} = 0,$$
and hence the $v$th term becomes zero.  This works for all $v$, $0 \leq v \leq \oCd$, so the coefficient of $z^{\oCd}$ in \eqref{eq:pureint} becomes zero.  
\end{proof}

We now choose $n$, $\yvec$, and $\cvec$ satisfying the hypothesis of Claim~\ref{clm:noint} for arbitrarily large, but fixed, $\oCd$.  Although a dimension-counting argument would suffice, there is a simple choice coming from combinatorics.
\begin{prop}[\cite{Lovasz}, Exercise 1.2.4]
Fix $\oCd \in \NN \cup \bracepair{0}$, $n := \oCd+ 2$, and for $1 \leq j \leq n$, let $y_j := j - 1$ and $\disp c_j := (-1)^{j-1} \binom{\oCd + 1}{j-1}$.  Then for all $\ell \in \bracepair{0, 1, \dotsc, \oCd}$,
$$\sum_{j = 1}^n c_j y_j^{\ell} = 0.$$ 
\label{prop:comb}
(Notice that the indices are shifted from \cite{Lovasz}).
\end{prop}


This allows us to fulfill the conditions of Claim~\ref{clm:noint} for any $\oCd$.  In particular, to fulfill the conditions of Corollary~\ref{cor:point}, we use the following case of Proposition~\ref{prop:comb}.

\begin{cor}
A set of solutions to \eqref{eq:pointcond} is given by $n := \ourintT + 2$, $\yvec = (0, 1, \dotsc, \ourintT + 1)$, and $\disp \cvec = \left( (-1)^{j - 1} \binom{\ourintT + 1}{j - 1} \right)_{j = 1}^{\ourintT + 2} $.  
\label{cor:sample}
\end{cor}

We will now show that the choices of $n$, $\yvec$, and $\cvec$ as above are sufficient to ensure that if $\ourvarT \in \bigcup_{k \in \NN} (2k-1, 2k)$, then the $z^{\ourvarT}$ coefficient is negative.

\subsection{Analyzing the coefficient of \texorpdfstring{$z^{\ourvarT}$ in $f$ }{z-super-t in f}} \label{subsect:tpows}

If $\ourvarT \not\in \NN$, we note that the coefficient of $z^{\ourvarT}$ in \eqref{eq:fullpowers} comes solely from products of one degree-$\ourvarT$ term and $n^2 - 2$ degree-$0$ terms in \eqref{eq:fullz}; no integer equals $\ourvarT$, so no product of degree-$0$ and degree-$1$ terms works, and the only other combination of terms with nonzero, but small enough, degree is one degree-$\ourvarT$ term and $n^2 - 2$ degree-$0$ terms.   Therefore, for all $\ourvarT > 0$, $a > 0$, the  coefficient of $z^{\ourvarT}$ in \eqref{eq:fullpowers} is $$\sum_{j = 1}^n \sum_{k = 1}^n c_j {c_k} \left( \sum_{\substack{ (p, q) \in \setn^2 \\ (p, q) \neq (j, k) }} a (y_p^2 + y_q^2)^{\ourvarT} \right).$$ 

As before, we attempt to make the summation independent of $(j, k)$ by adding and subtracting the $(j, k)$-th term in the sum, so the coefficient of $z^{\ourvarT}$ becomes

\begin{align*}
& \sum_{j = 1}^n \sum_{k = 1}^n a c_j {c_k}  \left(\left( \sum_{(p, q) \in \setn^2}  (y_p^2 + y_q^2)^{\ourvarT} \right) - (y_j^2 + y_k^2)^{\ourvarT} \right)  \\
& =  a \left( \sum_{(p, q) \in \setn^2}  (y_p^2 + y_q^2)^{\ourvarT} \right) \parenpair{ \sum_{j = 1}^n c_j }^2 - a  \sum_{j = 1}^n \sum_{k = 1}^n c_j {c_k} (y_j^2 + y_k^2)^{\ourvarT} .\\
\end{align*}

Now, assume that $\ourvarT \not\in \NN$ and that the values $\bracepair{y_j}_{j = 1}^n$ and $\bracepair{c_j}_{j = 1}^n$ are chosen so that the conditions in \eqref{eq:pointcond} hold.  Then in particular, $\disp \sum_{j = 1}^n c_j = 0$.  Hence, the first term becomes zero, and we are left with 

\[
- a  \sum_{j = 1}^n \sum_{k = 1}^n c_j {c_k} (y_j^2 + y_k^2)^{\ourvarT} .
\]

Our objective will be to show the following.
\begin{bigclm} Fix $\ourvarT > 0$, $\ourvarT \not\in \NN$, and write $\ourvarT = \ourintT + \ourfracT$ with $\ourintT \in \NN \cup \bracepair{0}$ and $\ourfracT \in (0, 1)$.  Fix $a > 0$.  For any $n \in \NN$, $\yvec = (y_1, \dotsc, y_n) \in \RR^n$ and $\cvec = (c_1, \dotsc, c_n) \in \RR^n$ such that the conditions in \eqref{eq:pointcond} hold,
\begin{equation}
\disp - a  \sum_{j = 1}^n \sum_{k = 1}^n c_j {c_k} (y_j^2 + y_k^2)^{\ourvarT}  
\label{eq:ztone}
\end{equation} 
 is nonnegative if $\ourintT$ is even, and is nonpositive if $\ourintT$ is odd.  Moreover, for any such $\ourvarT$ and $a$, a choice of $n$, $\yvec$ and $\cvec$ can be found such that the conditions in \eqref{eq:pointcond} hold and, in addition, \eqref{eq:ztone} is nonzero.  
\label{clm:zt}
\end{bigclm}
To proceed, we use a convenient integral representation for positive noninteger powers, such as the $t$-th power above.
\begin{lem}
If $\Re w > 0$ and $\ourvarR > 0$, $\ourvarR \not\in \NN$, then letting $\ourvarR  = \ourintR + \ourfracR$, where $\ourintR \in \NN \cup \bracepair{0}$, $\ourfracR \in (0, 1)$,
\begin{align}
\begin{aligned}
w^{\ourvarR } &= h(w; \ourvarR ) \\
 &:=(-1)^{\ourintR} \frac{\risingfac{\ourfracR}{\ourintR + 1}}{\Gamma(1 - \ourfracR )}  \int_0^{\infty} \left(\left( \sum_{\ell = 0}^{\ourintR} \frac{(-\lambda w)^{\ell}}{\ell !} \right) - e^{-\lambda w} \right)\,  \frac{d\lambda}{\lambda^{\ourintR + 1 + \ourfracR}}.
\end{aligned}\label{eq:wt}
\end{align}
\label{lem:wt}
\end{lem}

Note that \eqref{eq:wt} also holds when $w = 0$, if we define $0^{\ourvarR} = 0$.

The basic idea of the formula is not original: for example, Berg, Christensen, and Ressel present the formula for $0 < \ourvarR < 1$ in the proof of \cite{BCR}, Proposition 3.2.10, and then proceed to derive the formula for $1 < \ourvarR < 2$ in the proof of \cite{BCR}, Proposition 3.2.11.  \eqref{eq:wt} simply continues the pattern further.  As such, we defer the details of the proof of Lemma~\ref{lem:wt} to Section~\ref{sec:appendix}, and proceed to demonstrate the validity of Claim~\ref{clm:zt}, given the integral representation.  

\begin{proof}[Proof of Claim~\ref{clm:zt}]  Let $n \in \NN$, $\yvec = (y_1, \dotsc, y_n) \in \RR^n$, and $\cvec = (c_1, \dotsc, c_n) \in \RR^n$ satisfy \eqref{eq:pointcond}.  Then applying Lemma~\ref{lem:wt} to $\ourvarR = \ourvarT$ and $w = y_j^2 + y_k^2$, letting $\ourvarT = \ourintT + \ourfracT$\vspace{6 pt}, and letting  $B(\ourvarT)$  denote the positive constant 
$\disp \frac{\risingfac{\ourfracT}{\ourintT + 1}}{\Gamma(1 - \ourfracT )} $, we have that
\begin{align*}
& - a  \sum_{j = 1}^n \sum_{k = 1}^n c_j {c_k} (y_j^2 + y_k^2)^{\ourvarT}\\
& = (-1)^{\ourintT + 1} \, a  \, B(\ourvarT) \sum_{j = 1}^n \sum_{k = 1}^n c_j {c_k} \, \cdot \\
& \quad \quad \cdot \int_0^{\infty} \left( \sum_{\ell = 0}^{\ourvarT} \left( \frac{(-\lambda )^{\ell}(y_j^2 + y_k^2)^{\ell}}{\ell !} \right) - e^{-\lambda (y_j^2 + y_k^2)} \right)\,  \frac{d\lambda}{\lambda^{\ourintT + 1 + \ourfracT}} .
\end{align*}
Interchanging the sum in $j$ and $k$ and the integral, we get
\begin{align*}
(-1)^{\ourintT + 1} a\,  B(\ourvarT)   & \left( \int_0^{\infty}  \left( \sum_{\ell = 0}^{\ourintT} \frac{(-\lambda)^{\ell}}{\ell !}  \sum_{j = 1}^n \sum_{k = 1}^n c_j {c_k} (y_j^2 + y_k^2)^{\ell}  \, \right. \right. \\
&  \left. \left. \,  \quad  \quad \quad \quad \, -   \sum_{j = 1}^n \sum_{k = 1}^n c_j {c_k} e^{-\lambda (y_j^2 + y_k^2)} \,  \right) \frac{d\lambda}{\lambda^{\ourintT + 1 + \ourfracT}}   \right) .
\end{align*}
We now prove that if \eqref{eq:pointcond} holds, the term $\disp  \sum_{j = 1}^n \sum_{k = 1}^n c_j {c_k} (y_j^2 + y_k^2)^{\ell} $ zeroes out for all $\ell \in \bracepair{0, 1, \dotsc, \ourintT}$, and hence the initial terms of the integrand become zero.  We use the binomial formula to rewrite the sums, and we get, for the $\ell$-th double-sum,
\begin{align*}
&\sum_{j = 1}^n \sum_{k = 1}^n c_j {c_k} (y_j^2 + y_k^2)^{\ell}\\
& =\sum_{j = 1}^n \sum_{k = 1}^n c_j {c_k} \sum_{u = 0}^{\ell} \binom{\ell}{u} (y_j)^{2u} \cdot (y_k)^{2(\ell - u)}\\
&=  \sum_{u = 0}^{\ell} \binom{\ell}{u} \parenpair{ \sum_{j = 1}^n c_j y_j^{2u}} \parenpair{ \sum_{k = 1}^n c_j y_k^{2(\ell - u)}} .
\end{align*}
Since $\ell \leq \ourintT$, either $2u \leq \ourintT$ or $2(\ell - u) \leq \ourintT$, since $2u + 2(\ell - u) = 2 \ell \leq 2 \ourintT$.  Therefore, either the sum in $j$ or the sum in $k$ is 0, by \eqref{eq:pointcond}.  Therefore, the $\ell$th term becomes zero. Since this works for all $\ell$, $0 \leq \ell \leq \ourintT$, the initial terms of the integrand become zero.

Therefore, the expression in \eqref{eq:ztone} becomes
\[
 (-1)^{\ourintT + 1}\,  a\,B(\ourvarT) \left( - \int_0^{\infty}  \sum_{j = 1}^n \sum_{k = 1}^n c_j {c_k} e^{-\lambda (y_j^2 + y_k^2)} \,  \frac{d\lambda}{\lambda^{\ourintT + 1 + \ourfracT}}   \right);
\]
after simplifying, we have
\[
(-1)^{\ourintT} a \,B(\ourvarT) \left( \int_0^{\infty} \left(\sum_{j = 1}^n c_j e^{- \lambda y_j^2} \right)^2 \frac{d\lambda}{\lambda^{\ourintT + 1 + \ourfracT}}  \right).
\]

Since all terms are real, the integrand is clearly nonnegative, and $B(\ourvarT) = \disp \frac{\risingfac{\ourfracT}{\ourintT + 1}}{\Gamma(1 - \ourfracT )}$ is positive.  Therefore, the sign of the expression in \eqref{eq:ztone} depends only on $(-1)^{\ourintT}$, so it is nonnegative if $\ourintT$ is even and nonpositive if $\ourintT$ is odd.  

To ensure that strict positivity in the integral can occur, we note that the integral is strictly positive unless for Lebesgue-a.e. $\lambda > 0$,
\begin{equation}
\sum_{j = 1}^n c_j e^{- \lambda y_j^2} = 0. \label{eq:zerocond}
\end{equation}

By continuity in $\lambda$, we may assert that the integral is strictly positive unless \eqref{eq:zerocond} holds for all $\lambda > 0$.  Now, we do not immediately have that if \eqref{eq:zerocond} holds for all $\lambda > 0$, $c_j = 0$ for all $j$; for example, take $n = 2$, $y_2 = -y_1$, $c_2 = -c_1$.  The following statement, however, gives a sufficient condition to draw such a conclusion.
\begin{lem}
If $\yvec = (y_1, \dotsc, y_n) \in \RR^n$ is a set of distinct nonnegative real numbers, and if \eqref{eq:zerocond} holds for all $\lambda > 0$, then $c_j = 0$ for all $j$. \label{lem:biglinind}
\end{lem}
\begin{proof}
Without loss of generality, we may assume that $0 \leq y_1 < y_2 < \dotsb < y_n$.  Suppose that \eqref{eq:zerocond} holds, but that not all $c_j$ are zero; let $j_0$ be\vspace{3 pt} the smallest $j$ such that $c_j \neq 0$.  Then multiplying \eqref{eq:zerocond} by $\disp \frac{e^{\lambda y_{j_0}^2}}{c_{j_0}}$, we get that for all $\lambda > 0$,
\[
g(\lambda ; \yvec, \cvec) := 1 + \sum_{j = j_0 + 1}^n c_j e^{- \lambda ( y_j^2 - y_{j_0}^2)} = 0.
\]
Yet by $0 \leq y_{j_0}  < y_j$ for all $j > j_0$, $\disp \lim_{\lambda \to \infty} g(\lambda; \yvec, \cvec) = 1$, but $\disp \lim_{\lambda \to \infty} 0 = 0$.  Contradiction.
\end{proof}
In particular, then, for $\yvec$ and $\cvec$ chosen as  Corollary~\ref{cor:sample}, we have that $y_j = j -1$, so the $y_j$'s are distinct nonnegative real numbers.  By Lemma~\ref{lem:biglinind},
$\disp \sum_{j = 1}^n c_j e^{- \lambda y_j^2} \equiv 0$  if and only if $c_j = 0$ for all $j$, which does not hold for the choice of $c_j$ as in Corollary~\ref{cor:sample}.  Therefore, for an appropriate choice of $n$, $\bracepair{y_j}_{j = 1}^n$, and $\bracepair{c_j}_{j = 1}^n$, the integral is nonzero, so the expression in \eqref{eq:ztone} is strictly negative or positive, depending on the parity of $\ourintT$.  
\end{proof}

\subsection{Conclusion} \label{subsect:conclude}

We now finish the proof of Theorem 3.

\begin{proof}[Proof of Theorem 3]
Fix $\ourvarT \in \bigcup_{n \in \NN}(2k - 1, 2k)$, and write $\ourvarT = \ourintT + \ourfracT$,  where $\ourintT \in \NN$ is odd, and $\alpha \in (0, 1)$.  Also, fix $a > 0$.  By Corollary~\ref{cor:point} and Corollary~\ref{cor:sample}, for $n = \ourintT + 2$, $\yvec = (0, 1, \dotsc, \ourintT + 1)$, and\\ $\disp \cvec = \left( (-1)^{j - 1} \binom{\ourintT + 1}{j - 1} \right)_{j = 1}^{\ourintT + 2} $, \vspace{3 pt}all terms of degree less than $\ourvarT$ in \eqref{eq:fullz} have zero coefficients.  By Claim~\ref{clm:zt}, the same choice of $n$, $\yvec$, and $\cvec$  ensures that the coefficient of $z^{\ourvarT}$ in \eqref{eq:fullz} is negative.  Since all remaining terms are of degree larger than $\ourvarT$, we have that \eqref{eq:fullz} is of the form
\[
f\left( z;n, \yvec, \cvec \right) =  z^{\ourvarT} \left( \kappa + \epsilon(z) \right),
\]
where $\kappa <  0$ is the $z^{\ourvarT}$ coefficient, and since $r_m > t$ for all terms such that $b_m \neq 0$, $\disp \lim_{z \to 0^+} \epsilon(z) = 0$.  Then
$$\lim_{z \to 0^+} \left(\kappa + \epsilon(z)\right) = \kappa < 0.$$
Thus, for sufficiently small positive $z$,  $f$ is negative; hence, $K_{\ourvarT, a}$ is not positive-definite.
\end{proof}

To finish this section, we note that we still have not solved the specific question raised in \cite{MitSob}, namely the existence of negative eigenvalues for $\Kopmore$ in the case $t = 2$, $a$ near 0.  We therefore end with the following. \vspace{1ex}

\noindent
\textbf{Open Problem.}  Determine whether or not $\Kopmore$ admits a negative eigenvalue for $t = 2$ and $0 < a \leq 12 =: a_0(2)$.  

\section{Proof of Lemma~\ref{lem:wt}} \label{sec:appendix}

To prove Lemma~\ref{lem:wt}, we use the following $w$-dependent bound on the $L^1$-norm of the integrand in $h(w; \ourvarR)$.

\begin{lem}
If $\Re(w) \geq 0$, $\ourvarR > 0$, $\ourvarR \not\in \NN$, then writing $\ourvarR = \ourintR + \ourfracR$, where $\ourintR \in \NN \cup \bracepair{0}$ and $\ourfracR \in (0, 1)$,
\begin{equation}\int_0^{\infty} \left\vert \left( \sum_{\ell = 0}^{\ourintR} \frac{(-\lambda w)^{\ell}}{\ell !} \right) - e^{-\lambda w}  \right\vert \, \frac{d\lambda}{\lambda^{\ourintR + 1 + \ourfracR}}  \leq C(\ourvarR) \vert w \vert^{\ourvarR}, \label{eq:wbound}
\end{equation}
where $C(\ourvarR)$ is a constant depending on $\ourvarR$. \label{lem:wbound}
\end{lem}

\begin{proof}  If $w = 0$, both sides of \eqref{eq:wbound} are 0.  So assume that $w \neq 0$.  To take advantage of homogeneity, set $\mu = \vert w \vert \lambda$, $d\mu = \vert w \vert d\lambda$.  Letting $\disp \omega = \frac{w}{\vert w \vert} \in S^1 \cap \bracepair{\genercpx \in \CC: \Re \genercpx \geq 0}$, we get
\begin{align*}
& \int_0^{\infty} \left\vert \left( \sum_{\ell = 0}^{\ourintR} \frac{(-\mu \omega)^{\ell}}{\ell !} \right) - e^{-\mu \omega}  \right\vert \, \frac{d\mu/\vert w \vert}{(\mu / \vert w \vert)^{\ourintR + 1 + \ourfracR}} \\
& = \vert w \vert^{\ourvarR} \int_0^{\infty} \left\vert \left( \sum_{\ell = 0}^{\ourintR} \frac{(-\mu \omega)^{\ell}}{\ell !} \right) - e^{-\mu \omega}  \right\vert \, \frac{d\mu}{\mu^{\ourintR + 1 + \ourfracR}}.
\end{align*}
We have extracted the $\vert w \vert^{\ourvarR }$ term; now we need only show that for $\omega$ in $S^1 \cap \bracepair{\Re \genercpx \geq 0}$, the integral is bounded by a constant depending on $\ourvarR$ alone.   
We split the integral over $[0, \infty)$ into integrals over $[0, 1]$ and $[1, \infty)$. For the integral on $[0, 1]$, we recognize the initial sum as the first few terms of the Taylor series for $e^{- \mu \omega}$, and we get

\begin{align*}
&\int_0^{1} \left\vert \left( \sum_{\ell = 0}^{\ourintR} \frac{(- \mu \omega)^{\ell}}{\ell !} \right) - e^{- \mu \omega} \right\vert \,  \frac{d\mu}{\mu^{\ourintR + 1 + \ourfracR}}\\
 & = \int_0^{1} \left\vert - \left( \sum_{\ell = \ourintR + 1}^{\infty} \frac{(- \mu \omega)^{\ell}}{\ell !} \right) \right\vert \,  \frac{d\mu}{\mu^{\ourintR + 1 + \ourfracR}}\\
 & \leq \int_0^{1} \sum_{\ell = \ourintR + 1}^{\infty} \frac{1}{\ell !} \mu^{\ell - (\ourintR + 1 + \ourfracR)} \, d\mu \\
&=  \sum_{\ell = \ourintR + 1}^{\infty} \frac{1}{(\ell - (\ourintR + \ourfracR)) \cdot \ell!} \, .
\end{align*}

For the integral on $[1, \infty)$, we bound $e^{- \mu \omega}$ in absolute value by $1$ (by $\Re \omega \geq 0$) and get
\begin{align*}
& \int_{1}^{\infty} \vertpair{\left( \sum_{\ell = 0}^{\ourintR} \frac{(-\mu \omega)^{\ell}}{\ell !} \right) - e^{-\mu \omega}} \frac{1}{\mu^{\ourintR  + 1 + \ourfracR}} \, d\mu \\
 \leq &\int_{1}^{\infty} \left( \left( \sum_{\ell = 0}^{\ourintR}   \frac{1}{\ell !} \mu^{\ell - (\ourintR + 1 + \ourfracR)} \right) + \mu^{- (\ourintR + 1 + \ourfracR)} \right)  d\mu\\
 = &\left(\sum_{\ell = 0}^{\ourintR} \frac{-1}{(\ell - (\ourintR + \ourfracR)) \ell !} \right) + \frac{1}{(\ourintR + \ourfracR)}.
\end{align*}

Altogether, then, the integral is bounded by
\begin{align*}
 \left( \sum_{\ell =  0}^{\infty} \frac{1}{\vert\ell - (\ourintR + \ourfracR)\vert \cdot \ell!} \right) + \frac{1}{\ourintR + \ourfracR} .
\end{align*} 
Moreover, for all $\ell \in \NN \cup \bracepair{0}$, $\vert \ell - (\ourintR + \ourfracR) \vert \geq \min ( \ourfracR, 1 - \ourfracR) =: C_1(\ourfracR)$; hence, the term in parentheses is bounded above by
\[
(C_1(\ourfracR))^{-1} \sum_{\ell = 0}^{\infty} \frac{1}{\ell !} + \frac{1}{\ourintR + \ourfracR} = \frac{e}{C_1(\ourfracR)} + \frac{1}{\ourvarR}.
\]
Defining $\disp C(\ourvarR) := \frac{e}{C_1(\ourfracR)} + \frac{1}{\ourvarR}$, the proof is complete.
\end{proof}

\begin{proof}[Proof of Lemma~\ref{lem:wt}]
Since any noninteger $\ourvarR > 0$, can be written as $\ourvarR = \ourintR + \ourfracR$ for some $\ourintR \in \NN \cup \bracepair{0}$ and $\ourfracR \in (0, 1)$, it suffices to prove for each $\ourfracR \in (0, 1)$ that $\ourvarR = \ourintR + \ourfracR$ satisfies \eqref{eq:wt} for all $\ourintR$ in $\NN \cup \bracepair{0}$.  

We therefore induct on $\ourintR$.  We note that if $\ourvarR = 0 + \ourfracR$, implying $0 < \ourvarR < 1$, the fact is well known; see \cite{BCR}, p. 78.

Suppose that for some $\ourintR \in \NN \cup \bracepair{0}$, \eqref{eq:wt} is true for $\ourvarR = \thisp := \ourintR + \ourfracR$.  We will prove \eqref{eq:wt} true for  $\ourvarR = \thisq: = (\ourintR + 1) + \ourfracR$.  We first demonstrate that for $\ourvarR = q$, both sides of \eqref{eq:wt} have the same derivative, namely $\thisq \cdot h(w; \thisp)$.
 
The derivative of $w^{\thisq}$ is $\thisq w^{\thisq-1} = \thisq w^\thisp$.  By the inductive hypothesis, $\thisq w^{\thisp} =  \thisq \cdot h(w; \thisp)$.
On the other hand, the derivative of the integrand in $h(w; \thisq)$ is 
\begin{align*}
&\frac{d}{dw} \left(\left( \sum_{\ell = 0}^{\ourintR + 1} \frac{(-\lambda w)^{\ell}}{\ell !} \right) - e^{-\lambda w} \right)\cdot \frac{1}{\lambda^{\ourintR + 2 + \ourfracR}}\\
& =\left( \left( \sum_{\ell = 1}^{\ourintR + 1} \frac{(- \lambda w)^{\ell - 1} (- \lambda)}{\ell !} \right) - (- \lambda)e^{- \lambda w} \right) \frac{1}{\lambda^{\ourintR + 2 + \ourfracR}} \, .
\end{align*}
Reindexing by $u = \ell - 1$, we get
\[\left( \left( \sum_{u = 0}^{\ourintR} \frac{(- \lambda w)^{u} }{\ell !} \right) - e^{- \lambda w} \right)\cdot \left( \frac{-1}{\lambda^{\ourintR + 1 + \ourfracR}} \right)\, .
\]
The absolute value of this expression is the integrand in \eqref{eq:wbound}, with $s = p$.  Thus, by Lemma~\ref{lem:wbound}, we have $L^1(\RR^+, d\lambda)$ convergence, uniform in $w$ for $w \in \bracepair{\vert \genercpx \vert \leq M} \cap \bracepair{\Re \genercpx \geq 0}$, for any $M > 0$.  Hence, we can differentiate under the integral sign on the right-hand side, and we get
\begin{align*}
&\frac{d}{dw} ( h(w; q) )\\
& = (-1)^{(\ourintR + 1)} \frac{\risingfac{\ourfracR}{\ourintR + 2}}{\Gamma(1 - \ourfracR)}  \int_0^{\infty}  \left( \left( \sum_{u = 0}^{\ourintR} \frac{(- \lambda w)^{u}  }{u !} \right) -  e^{- \lambda w} \right)\cdot \left( \frac{-d\lambda}{\lambda^{\ourintR + 1 + \ourfracR}} \right)\\
& = (-1)^{\ourintR} \frac{( \ourintR + 1 + \ourfracR)\risingfac{\ourfracR}{\ourintR + 1}}{\Gamma(1 - \ourfracR)} \int_0^{\infty}  \left(\left( \sum_{u = 0}^{\ourintR} \frac{(- \lambda w)^{u}}{u !} \right) -  e^{-\lambda w} \right)\,  \frac{d\lambda}{\lambda^{\ourintR + 1 + \ourfracR}}.
\end{align*}
By definition, this is  $(\ourintR + 1 + \ourfracR) \cdot h(w; \thisp) = \thisq \cdot h(w; \thisp)$.  Yet the derivative of $w^{\thisq}$ was $\thisq \cdot h(w; \thisp)$.  So we see that $w^{\thisq}$ and $h(w; \thisq)$ have the same derivative.  Hence, they are equal up to a constant on the domain of mutual definition: $h(w; \thisq) - w^{\thisq} = D$ for all $w \in \bracepair{\Re \genercpx > 0}$.  

To show that $D = 0$, we note that as $w$ approaches $0$ along the positive real axis, $w^q \to 0$.  More importantly, by Lemma~\ref{lem:wbound}, as $w$ approaches $0$ along the positive real axis, $h(w; q) \to 0$ as well.  Hence,
\[D = \lim_{\substack{w \to 0^+ \\ 
w \text{ real}}} D = \lim_{\substack{w \to 0^+ \\ 
w \text{ real}}} h(w; \thisq) - w^{\thisq} = 0.  \]
\end{proof}

\bibliographystyle{plain}
\bibliography{AnIntOpPosDef.CEJM.refs}{}

\begin{thebibliography}{1}

\bibitem{Adduci}
James Adduci.
\newblock {\em Perturbations of self-adjoint operators with discrete spectrum}.
\newblock PhD thesis, Ohio State University, Columbus, Ohio, 2011.

\bibitem{BCR}
Christian Berg, Jens Peter~Reus Christensen, and Paul Ressel.
\newblock {\em Harmonic analysis on semigroups: theory of positive definite and
  related functions}, volume 100 of {\em Graduate Texts in Mathematics}.
\newblock Springer-Verlag, New York, 1984.

\bibitem{KrotChubOne}
Pavel Krotkov and Andrey~V. Chubukov.
\newblock Non-{F}ermi liquid and pairing in electron-doped cuprates.
\newblock {\em Phys. Rev. Lett.}, 96:107002, Mar 2006.

\bibitem{KrotChubTwo}
Pavel Krotkov and Andrey~V. Chubukov.
\newblock Theory of non-{F}ermi liquid and pairing in electron-doped cuprates.
\newblock {\em Phys. Rev. B}, 74:014509, Jul 2006.

\bibitem{Lovasz}
L{\'a}szl{\'o} Lov{\'a}sz.
\newblock {\em Combinatorial problems and exercises}.
\newblock AMS Chelsea Publishing, Providence, RI, second edition, 2007.

\bibitem{MitAn}
Boris Mityagin.
\newblock An anisotropic integral operator in high temperature
  superconductivity.
\newblock {\em Israel J. Math.}, 181:1--28, 2011.

\bibitem{MitSob}
Boris~S. Mityagin and Alexander~V. Sobolev.
\newblock A family of anisotropic integral operators and behavior of its
  maximal eigenvalue.
\newblock {\em J. Spectr. Theory}, 1(4):443--460, 2011.

\end{thebibliography}

\end{document}